\documentclass[10pt,reqno]{amsart}

\usepackage[T1]{fontenc}
\usepackage[utf8]{inputenc}
\usepackage{lmodern}
\usepackage{microtype}
\usepackage{amsmath,amssymb,mathtools}
\usepackage[a4paper,margin=26mm]{geometry}
\usepackage{xcolor}
\usepackage[colorlinks=true,linkcolor=blue!45!black,
  citecolor=blue!45!black,urlcolor=blue!45!black]{hyperref}

\newtheorem{theorem}{Theorem}
\newtheorem{proposition}[theorem]{Proposition}

\newtheorem{corollary}[theorem]{Corollary}
\theoremstyle{remark}
\newtheorem{remark}[theorem]{Remark}

\DeclareMathOperator{\Spin}{Spin}

\newcommand{\R}{\mathbb R}
\newcommand{\Sph}{\mathbb S}
\newcommand{\SM}{\mathbb S M}
\newcommand{\tr}{\operatorname{tr}}
\newcommand{\diver}{\operatorname{div}}
\newcommand{\Ric}{\operatorname{Ric}}
\newcommand{\vol}{\operatorname{vol}}
\newcommand{\Rad}{\operatorname{Rad}}
\newcommand{\Id}{\operatorname{Id}}
\newcommand{\dd}{\,\mathrm d}
\newcommand{\ip}[2]{\left\langle #1,#2\right\rangle}
\newcommand{\norm}[1]{\left\lvert #1\right\rvert}
\newcommand{\Norm}[1]{\left\lVert #1\right\rVert}

\title[Spacetime mean curvature of DEC spin fill-ins]
{Dirac eigenvalues and spacetime mean curvature estimates for DEC spin fill-ins}

\author{Simon Raulot}
\address{Universit\'e de Rouen Normandie, CNRS, Normandie Universit\'e,
LMRS UMR 6085, F-76000 Rouen, France}
\email{simon.raulot@univ-rouen.fr}
\date{\today}

\begin{document}

\begin{abstract}
	We establish spacetime counterparts of fill-in inequalities for spin initial data sets satisfying the dominant energy condition in dimensions three through seven, controlling the minimum and the integral of the norm of the spacetime mean curvature vector by intrinsic boundary data. The proof relies on a Reilly--MIT estimate for complete spin manifolds with compact boundary. As further consequences, under suitable nonnegativity assumptions, Gromov's hyperspherical-radius fill-in inequality and B\"ar's recent total mean-curvature estimate extend to complete, possibly noncompact spin manifolds with compact boundary.
\end{abstract}

\maketitle


\section{Introduction and main results}


A basic question in the geometry of initial data sets is to determine to what extent the dominant energy condition controls the extrinsic geometry of the boundary. More specifically, let $(\Sigma^{n-1},\gamma)$ be a closed oriented Riemannian manifold endowed with a spin structure, and let $H,P\in C^\infty(\Sigma)$.
One may ask whether there exists a compact connected spin initial data set $(\Omega^n,g,K)$ whose induced boundary metric is $\gamma$, whose induced spin structure is the prescribed one, and such that
\[
H_g=H,\qquad \tr_\Sigma K=P,\qquad \mu\geq\norm J_g.
\]
Here $H_g$ denotes the mean curvature of $\partial\Omega$ with respect to the outward unit normal, with the convention that it is positive on Euclidean balls, and $\tr_\Sigma K$ is the trace of the restriction of $K$ to $T\Sigma$. The energy and momentum densities $\mu$ and $J$ are given by
\[
2\mu=R_g+(\tr_gK)^2-\norm K_g^2,
\qquad
J=\diver_g\bigl(K-(\tr_gK)g\bigr).
\]
Here and throughout, the spin structure on $\Sigma$ is fixed, and every boundary identification $\partial\Omega\simeq\Sigma$ is required to preserve it. This compatibility will be understood, and spin structures will therefore be omitted from the notation. We refer to $(\Sigma,\gamma,H,P)$ as a spin boundary data set and call any such
$(\Omega,g,K)$ a \emph{DEC spin fill-in} of these data. 

When $K\equiv0$, this definition reduces to the usual notion of a spin fill-in with nonnegative scalar curvature. More generally, taking $K=cg$ with $c$ constant recovers the Riemannian fill-in problem with scalar curvature bounded from below. For these fill-in problems, see \cite{Gromov,ShiTam,Jauregui,MantoulidisMiao,Miao, ShiWangWei,BTW,AmmannLockman,BaerTotal}.

Unlike in the time-symmetric fill-in problem, the natural boundary quantity is not the mean curvature alone. If \(H>|P|\), the boundary data are said to be \emph{untrapped}, and the associated spacetime mean-curvature vector is spacelike with norm
\[
\norm{\mathcal H}=\sqrt{H^2-P^2}.
\]
It is then natural to ask whether the existence of a DEC spin fill-in imposes upper bounds, depending only on the intrinsic spin geometry of the boundary, for the minimum and the integral of $\norm{\mathcal H}$.

The minimum of $\norm{\mathcal H}$ is naturally related to the spectrum of the boundary Dirac operator. For each connected component $\Sigma_j$, let $D^{\Sigma_j}$ denote its intrinsic Dirac operator and
set
\[
\lambda(\Sigma_j)
:=\min\bigl\{|\lambda|:\lambda\in\operatorname{Spec}(D^{\Sigma_j})\bigr\}.
\]
Thus, $\lambda(\Sigma_j)=0$ if and only if $D^{\Sigma_j}$ has a nontrivial kernel.

In dimension three, the author proved in \cite{Raulot} that, if a connected untrapped spin boundary data set admits a DEC spin fill-in without apparent horizons in its interior, then
\[
\lambda(\Sigma)
\geq\frac12\min_{\Sigma}\norm{\mathcal H}.
\]
Here a closed two-sided hypersurface $S\subset\operatorname{int}\Omega$ is called an \emph{apparent horizon} if, for one choice of unit normal,
\[
H_S+\tr_S K\equiv0
\qquad\text{or}\qquad
H_S-\tr_S K\equiv0.
\]
This condition is independent of the choice of unit normal. The assumption excluding apparent horizons was used to ensure the existence of a smooth solution of Jang's equation on the entire fill-in. A first contribution of the present paper is to show that this assumption is unnecessary. It also extends the estimate to dimensions $3\leq n\leq7$, the regularity range relevant to the construction of the Jang graph. More precisely, we prove the following.
\begin{theorem}\label{thm:initial-spectral}
	Let $3\leq n\leq7$, and let $(\Omega^n,g,K)$ be a DEC spin fill-in of an untrapped spin boundary data set $(\Sigma,\gamma,H,P)$. Then, for every connected component $\Sigma_j$, the Dirac operator $D^{\Sigma_j}$ has trivial kernel and
	\begin{equation}\label{eq:initial-spectral}
		\lambda(\Sigma_j)
		\geq\frac12\min_{\Sigma_j}\norm{\mathcal H}.
	\end{equation}
	If equality holds in \eqref{eq:initial-spectral} for some component, then $\Sigma$ is connected and $(\Omega,g)$ admits a local isometric immersion, as a spacelike hypersurface with second fundamental form $K$, into a Ricci-flat Lorentzian product whose Riemannian factor carries a nonzero parallel spinor. Moreover,
	$\norm{\mathcal H}$ is constant on $\Sigma$.
\end{theorem}
When equality holds and $n=3$, the ambient product is locally isometric to Minkowski spacetime, since every three-dimensional Ricci-flat Riemannian metric is flat.

The estimate in Theorem~\ref{thm:initial-spectral} combines directly with B\"ar's upper bound involving the hyperspherical radius $\Rad (\Sigma_j)$ of $\Sigma_j$. More precisely, B\"ar proved in \cite{BaerRadius} that
\begin{equation}\label{eq:Baer-radius}
	\lambda(\Sigma_j)
	\leq
	\frac{n-1}{2\Rad(\Sigma_j)},
\end{equation}
where $\Rad (\Sigma_j)$ is defined as the supremum of all $r>0$ such that there exists a $1$-Lipschitz map \(\Sigma_j\longrightarrow\Sph^{n-1}(r)\) of nonzero degree. Combining this estimate with Theorem~\ref{thm:initial-spectral} gives the following spacetime counterpart of Gromov's fill-in inequality
\cite[p.~110]{Gromov}.

\begin{corollary}\label{cor:radius}
Under the assumptions of Theorem~\ref{thm:initial-spectral}, every connected component $\Sigma_j$ satisfies
\[
 \min_{\Sigma_j}\norm{\mathcal H}\leq \frac{n-1}{\Rad(\Sigma_j)}.
\]
\end{corollary}

This contains two familiar Riemannian fill-in inequalities in a single statement. Indeed, if $K=cg$ for a constant $c\in\R$, then the dominant
energy condition is equivalent to $R_g\geq-n(n-1)c^2$, while $P=(n-1)c$. On every boundary component, Corollary~\ref{cor:radius} therefore gives
\[
 \min_{\Sigma_j} H
 \leq(n-1)\sqrt{c^2+
 \Rad (\Sigma_j)^{-2}}.
\]
Within the dimensional range considered here, the case $c=0$ recovers the preceding spin fill-in estimate of Gromov, which also follows from B\"ar's spectral estimate \cite{BaerRadius}; its rigidity statement was proved by Cecchini, Hirsch, and Zeidler \cite{CHZ}. The case $c=1$ is the inequality of Brendle, Tsiamis, and Wang \cite{BTW}, whose rigidity statement was recently proved by Ammann and Lockman \cite{AmmannLockman}.

Gromov also conjectured \cite[p.~232]{Gromov} that the total mean curvature of a fill-in can be bounded from above in terms of the intrinsic boundary geometry and a lower scalar-curvature bound. B\"ar recently proved the corresponding estimate for spin fill-ins, with the constant also depending on a lower bound for the boundary
mean curvature \cite{BaerTotal}. His argument can be combined with the Jang deformation used here to give a spacetime version in which the mean curvature is replaced by $\norm{\mathcal H}$. For pure-trace data $K=cg$, Theorem~\ref{thm:total} also recovers B\"ar's Riemannian total mean-curvature estimate in the untrapped regime.
\begin{theorem}\label{thm:total}
Let $3\leq n\leq7$ and let $(\Sigma^{n-1},\gamma)$ be a closed Riemannian spin manifold. For every connected component $\Sigma_j$ there is a constant
$C(\Sigma_j)$, depending only on the intrinsic Riemannian spin geometry of $\Sigma_j$, such that every untrapped spin boundary data set
$(\Sigma,\gamma,H,P)$ admitting a DEC spin fill-in satisfies
\begin{equation}\label{eq:total}
 \int_{\Sigma_j}\norm{\mathcal H}
 \leq C(\Sigma_j).
\end{equation}
\end{theorem}

In dimension three, suppose in addition that $(\Sigma,\gamma)$ is a topological sphere with positive Gauss curvature. Let $H_0>0$ denote the mean curvature of its convex isometric embedding into $\R^3$, which is unique up to rigid motions by the Weyl embedding theorem. The positivity of the Liu--Yau quasi-local mass \cite[Theorem~1]{LiuYau} gives
\[
\int_\Sigma\norm{\mathcal H}
\leq
\int_\Sigma H_0.
\]
Thus $C(\Sigma)=\int_\Sigma H_0$ is an explicit intrinsic choice in Theorem~\ref{thm:total}. Beyond this special case, Theorem~\ref{thm:total} yields such an intrinsic, though generally nonexplicit, bound for arbitrary untrapped spin boundaries of DEC fill-ins in dimensions $3\leq n\leq7$, without any assumption on their topology or intrinsic curvature.

The analytic part of the proof rests on the general framework developed by B\"ar and Ballmann for Dirac-type operators on complete manifolds with compact boundary \cite{BaerBallmann}. Indeed, for \(\varepsilon>0\), the imaginary shift $D-i\varepsilon$ is coercive at infinity, so their Fredholm theory applies to the shifted MIT boundary problem on the complete Jang graph; the corresponding Green identity then gives invertibility. This avoids both weighted Sobolev spaces and spectral assumptions on the ends. A cutoff argument subsequently extends the modified spinorial Reilly formula to the resulting spinors and allows us to pass to the limit as $\varepsilon$ tends to zero. Section~\ref{BoundaryEstimates} proves the general estimates for complete spin manifolds with compact boundary using this Reilly--MIT argument, while Section~\ref{sec:jang} applies it to the capillarity limit of Jang's equation.


\section{Boundary estimates on complete spin manifolds}\label{BoundaryEstimates}


Let $(M^n,h)$ be a connected complete spin manifold with nonempty compact smooth boundary
\[
\Sigma=\Sigma_1\sqcup\cdots\sqcup\Sigma_\ell,
\]
and let $X$ be a smooth vector field on $M$. Set
\[
	\mathcal E:=R_h-2\norm{X}_h^2+2\diver_hX
\]
and
\[
	\beta:=H-h(X,N),
\]
where $N$ is the outward unit normal and $H$ is the corresponding mean curvature.

The following theorem is the analytic core of the spectral part of the paper. It extends the compact estimate of \cite[Theorem~2]{Raulot} to complete manifolds with compact boundary, as required for the capillarity limit of Jang's equation. In particular, when $X=0$, it extends the boundary eigenvalue estimate of Hijazi--Montiel--Zhang \cite[Theorem~6]{HMZ} from compact spin manifolds with boundary to complete spin manifolds with compact boundary, without any assumption on the ends. A related $\Spin^c$ extension, which allows the boundary to be noncompact but assumes bounded geometry and coercivity at infinity, was obtained by Gro{\ss}e and Nakad
\cite[Theorem~1.2]{GrosseNakad}.

\begin{theorem}\label{thm:complete-spectral}
	Let $(M^n,h)$ be a connected complete spin manifold with nonempty compact smooth boundary \(\Sigma\) and let $X$ be a smooth vector field on $M$. Assume that $\mathcal E\geq0$ on $M$ and $\min_{\Sigma_j}\beta>0$ for every $j$. Then, for every connected component $\Sigma_j$, the intrinsic Dirac operator $D^{\Sigma_j}$ has trivial kernel and
	\begin{equation}\label{eq:complete-eigenvalue-estimate}
		\lambda(\Sigma_j)
		\geq
		\frac12\min_{\Sigma_j}\beta.
	\end{equation}
	If equality holds for some component, then $\Sigma$ is connected, $M$ is compact, $X\equiv0$, and $(M,h)$ carries a nonzero parallel spinor. Moreover, $\beta=H$ is constant on $\Sigma$ and $h$ is Ricci-flat. In particular, \eqref{eq:complete-eigenvalue-estimate} is strict whenever $M$ is noncompact.
\end{theorem}

Combining the case $X=0$ of Theorem~\ref{thm:complete-spectral} with B\"ar's hyperspherical-radius estimate \eqref{eq:Baer-radius} gives the following complete-manifold
extension of Gromov's fill-in inequality.

\begin{corollary}\label{cor:complete-radius}
	Let $(M^n,h)$ be a connected complete spin manifold with nonempty compact smooth boundary \(\Sigma\). Assume that $R_h\geq0$ on $M$ and $H>0$ on $\Sigma$. Then every connected component $\Sigma_j$ satisfies
	\[
	\min_{\Sigma_j}H
	\leq
	\frac{n-1}{\Rad(\Sigma_j)}.
	\]
	Moreover, the inequality is strict if $M$ is noncompact.
\end{corollary}

The same analytic method also gives the following total modified mean-curvature estimate.
\begin{theorem}\label{thm:complete-total}
	Let $(M^n,h)$ be a connected complete spin manifold with nonempty compact smooth boundary \(\Sigma\) and let $X$ be a smooth vector field on $M$. Assume that $\mathcal E\geq0$ on $M$ and $\beta\geq0$ on $\Sigma$. Then, for every connected component $\Sigma_j$, there exists a constant $C(\Sigma_j)$, depending only on its intrinsic Riemannian spin geometry, such that
	\[
		\int_{\Sigma_j}\beta
		\leq
		C(\Sigma_j).
	\]
\end{theorem}

Taking $X=0$ in Theorem~\ref{thm:complete-total} gives the following complete-manifold extension of the case $R_h\geq0$ and $H\geq0$ of
\cite[Theorem~1]{BaerTotal}.
\begin{corollary}\label{cor:complete-total}
	Let $(M^n,h)$ be a connected complete spin manifold with nonempty compact smooth boundary \(\Sigma\). Assume that $R_h\geq0$ on $M$ and $H\geq0$ on $\Sigma$.
	Then, for every connected component $\Sigma_j$, there exists a constant $C(\Sigma_j)$, depending only on its intrinsic Riemannian spin geometry, such that
	\[
	\int_{\Sigma_j}H
	\leq
	C(\Sigma_j).
	\]
\end{corollary}

The proofs of Theorems~\ref{thm:complete-spectral} and \ref{thm:complete-total} rely on the same analytic ingredients. We first solve the shifted MIT boundary problem and establish the corresponding modified Reilly formula. The spectral and total estimates then follow from two different choices of boundary spinors.

\subsection{The shifted MIT boundary problem}

Let $\SM$ denote the complex spinor bundle of $(M,h)$. Its Hermitian product, taken to be linear in the first argument, spin connection, Clifford multiplication, and Dirac operator are denoted by $\langle\cdot,\cdot\rangle$, $\nabla$, $c$, and $D$, respectively. Along $\Sigma$, we work with the restricted bundle $\mathbf S:=\SM|_\Sigma$. Let $\mathcal D^\Sigma$ denote the adapted Dirac operator acting on sections of $\mathbf S$. Under the standard identifications, it is the intrinsic Dirac operator $D^\Sigma$ when $n-1$ is even, and $D^\Sigma\oplus(-D^\Sigma)$ when $n-1$ is odd. Consequently, its spectrum is symmetric in either case and, for every connected component $\Sigma_j$,
\[
\lambda(\Sigma_j)=\min\bigl\{|\lambda|:
\lambda\in\operatorname{Spec}(\mathcal D^{\Sigma_j})\bigr\}.
\]
Consider the fiberwise self-adjoint involution $s:=ic(\nu)$ where $\nu=-N$ is the inward unit normal and the associated orthogonal projections
\[
	P^\pm=\frac12(\Id\pm s).
\]
We write $\mathbf S^\pm:=\operatorname{im}P^\pm$. Since $s$ anticommutes with $\mathcal D^\Sigma$, the operator $\mathcal D^\Sigma$ interchanges $\mathbf S^+$ and $\mathbf S^-$; equivalently,
\begin{equation}\label{eq:swap}
	\mathcal D^\Sigma P^\pm=P^\mp\mathcal D^\Sigma.
\end{equation} 
Following B\"ar and Ballmann \cite{BaerBallmann}, we introduce the maximal domain
\[
\operatorname{dom}D_{\max}
=
\left\{
\psi\in L^2(M,\SM):
D\psi\in L^2(M,\SM)
\text{ in the distributional sense}
\right\}
\]
and the space
\[
H_D^1(M,\SM)
:=
\operatorname{dom}D_{\max}
\cap H^1_{\mathrm{loc}}(M,\SM),
\]
where the local regularity is understood up to the boundary. If $\chi\in C_c^\infty(M)$ is equal to one on a neighborhood of $\Sigma$, this space is endowed with the norm
\[
\Norm{\psi}_{H_D^1}^2
:=
\Norm{\chi\psi}_{H^1(M)}^2
+\Norm{\psi}_{L^2(M)}^2
+\Norm{D\psi}_{L^2(M)}^2.
\]
Different choices of $\chi$ give equivalent norms. Since $(M,h)$ is complete, the spin Dirac operator and its formal adjoint are complete in the sense of B\"ar and Ballmann by \cite[Theorem~1.2]{BaerBallmann}. By \cite[Theorem~6.7]{BaerBallmann}, every $\psi\in H_D^1(M,\SM)$ has a well-defined trace $\psi|_\Sigma\in H^{1/2}(\Sigma,\mathbf S)$. To lighten notation, we shall not distinguish between a spinor and its boundary trace in expressions over $\Sigma$. Green's formula then reads
\begin{equation}\label{eq:Green}
	\int_M
	\bigl(
	\langle D\psi,\eta\rangle
	-
	\langle\psi,D\eta\rangle
	\bigr)
	=
	-\int_\Sigma
	\langle c(\nu)\psi,\eta\rangle
\end{equation}
for all $\psi,\eta\in H_D^1(M,\SM)$. Taking $\eta=\psi$ in \eqref{eq:Green} and using $s=ic(\nu)$ gives
\begin{equation}\label{eq:Green-imaginary}
	2\operatorname{Im}
	\int_M\langle D\psi,\psi\rangle
	=
	\Norm{P^+\psi}_{L^2(\Sigma)}^2
	-
	\Norm{P^-\psi}_{L^2(\Sigma)}^2.
\end{equation}

We first consider the MIT boundary problem for the shifted operator $D-i\varepsilon$. Although the same problem appears in \cite{BaerTotal}, the shift plays a different role here. In B\"ar's compact setting, it is calibrated to the scalar-curvature lower bound and contributes to the final estimate. In the present complete, possibly noncompact setting, it is instead an auxiliary analytic parameter which guarantees coercivity at infinity. Its value is unrelated to the geometric data, and it will eventually be allowed to tend to zero.

Set
\[
B^\pm:=H^{1/2}(\Sigma,\mathbf S^\pm)
\]
and, following the notation of \cite[Remark~8.1]{BaerBallmann}, write
\[
H_D^1(M,\SM;B^\pm)
:=
\left\{
\psi\in H_D^1(M,\SM):
 \psi|_\Sigma\in B^\pm
\right\}.
\]
Thus $H_D^1(M,\SM;B^-)$ is the domain corresponding to the homogeneous condition $P^+(\psi|_\Sigma)=0$, whereas $H_D^1(M,\SM;B^+)$ corresponds to $P^-(\psi|_\Sigma)=0$.

The following proposition extends \cite[Proposition~1]{BaerTotal} to complete manifolds and records the identity used in \cite[Lemma~1]{BaerTotal}.
\begin{proposition}
	For every $\varepsilon>0$ and $\varphi\in H^{1/2}(\Sigma,\mathbf S^+)$, there exists a unique $\Psi_\varepsilon\in H_D^1(M,\SM)$ such that
	\begin{equation}\label{eq:MIT-problem}
		(D-i\varepsilon)\Psi_\varepsilon=0,
		\qquad
		P^+(\Psi_\varepsilon|_\Sigma)=\varphi.
	\end{equation}
	It is smooth when $\varphi$ is smooth and satisfies
	\begin{equation}\label{eq:balance}
		\Norm{P^- \Psi_\varepsilon}_{L^2(\Sigma)}^2
		=
		\Norm{\varphi}_{L^2(\Sigma)}^2
		-
		2\varepsilon
		\Norm{\Psi_\varepsilon}_{L^2(M)}^2.
	\end{equation}
	The analogous statement holds for $D+i\varepsilon$ with prescribed $P^-$ trace and the two signs interchanged.
\end{proposition}

\begin{proof}
    The operators $D\mp i\varepsilon$ have the same principal symbol, maximal domain, and $H_D^1$ space as $D$, with equivalent norms; in particular, they are complete.
	
	The spaces $B^\pm$ define local elliptic boundary conditions by \cite[Example~7.26]{BaerBallmann}. Since a zero-order shift does not change the principal symbol, they are also elliptic for $D\mp i\varepsilon$.
	
	For every $\psi\in C_{cc}^\infty(M,\SM)$, where the subscript $cc$ means that the compact support is contained in the interior of $M$, formal self-adjointness of $D$ gives
	\[
		\Norm{(D\mp i\varepsilon)\psi}_{L^2(M)}^2
		=
		\Norm{D\psi}_{L^2(M)}^2
		+
		\varepsilon^2\Norm{\psi}_{L^2(M)}^2.
	\]
	Hence $D-i\varepsilon$ and its formal adjoint $D+i\varepsilon$ are coercive at infinity. Corollary~8.6 of \cite{BaerBallmann} therefore shows that
	\[
	\mathcal A_{\varepsilon,+}
	:=
	(D-i\varepsilon)_{B^-}
	:
	H_D^1(M,\SM;B^-)
	\longrightarrow L^2(M,\SM)
	\]
	is Fredholm.
	
	Since $B^\pm\subset H^{1/2}(\Sigma,\mathbf S)$, the maximal and $H_D^1$ realizations coincide by \cite[Lemma~7.3]{BaerBallmann}. The description of adjoint boundary
	conditions in \cite[Section~7.2]{BaerBallmann} gives
	\[
	\mathcal A_{\varepsilon,+}^*
	=
	(D+i\varepsilon)_{B^+}.
	\]
	Indeed, $c(\nu)$ preserves $\mathbf S^-$, so the annihilator of $c(\nu)B^-$ in the boundary pairing is $B^+$. Equivalently, the adjoint boundary condition is $P^-(\eta|_\Sigma)=0$.
	
	Let $\psi\in\ker\mathcal A_{\varepsilon,+}$. Then
	\[
	D\psi=i\varepsilon\psi,
	\qquad
	P^+(\psi|_\Sigma)=0.
	\]
	Since $\psi\in H_D^1(M,\SM)$, we may apply \eqref{eq:Green-imaginary} and obtain
	\[
	2\varepsilon\Norm{\psi}_{L^2(M)}^2
	=
	-\Norm{P^-\psi}_{L^2(\Sigma)}^2.
	\]
	Thus $\psi=0$. The same argument applied to $D+i\varepsilon$ with the homogeneous $P^-$ boundary condition shows that $\ker\mathcal A_{\varepsilon,+}^*=\{0\}$. Hence the Fredholm operator $\mathcal A_{\varepsilon,+}$ has trivial kernel and cokernel, and is therefore an isomorphism.
	
	Given $\varphi\in H^{1/2}(\Sigma,\mathbf S^+)$, choose an $H^1$ extension $\widetilde\varphi$ of $\varphi$, supported in a compact collar of $\Sigma$. Since
	$(D-i\varepsilon)\widetilde\varphi\in L^2(M,\SM)$, there is a unique $u\in H_D^1(M,\SM;B^-)$ satisfying
	\[
	(D-i\varepsilon)u
	=
	-(D-i\varepsilon)\widetilde\varphi.
	\]
	Then
	\[
	\Psi_\varepsilon=u+\widetilde\varphi
	\]
	is the unique solution of \eqref{eq:MIT-problem}. If $\varphi$ is smooth, smoothness up to the boundary follows from \cite[Theorem~7.17 and Proposition~7.24]{BaerBallmann}, together with interior elliptic regularity.
	
	Finally, applying \eqref{eq:Green-imaginary} to $\Psi_\varepsilon$ gives \eqref{eq:balance}. The proof for $D+i\varepsilon$ with prescribed $P^-$ trace is identical, with the signs interchanged.
\end{proof}

\subsection{The modified Reilly identity}

For every smooth spinor compactly supported up to the boundary, the integrated Schr\"odinger--Lichnerowicz formula and the computation in \cite[Proposition~3]{Raulot} give
\begin{align}\label{eq:compact-Reilly}
 \mathcal B_\Sigma(\psi)
 =\frac12\int_M\norm{\nabla\psi}^2
  +\frac12\int_M
     \norm{\nabla\psi+X^\flat\otimes\psi}^2
  +\frac14\int_M\mathcal E\norm\psi^2
  -\int_M\norm{D\psi}^2,
\end{align}
where we let 
\begin{align*}
	\mathcal B_\Sigma(\psi)
	:=\operatorname{Re}\int_\Sigma
	\left(\ip{\mathcal D^\Sigma\psi}{\psi}
	-\frac\beta2\norm\psi^2\right).
\end{align*}
We use the same notation for the continuous extension of $\mathcal B_\Sigma$ to $H_D^1(M,\SM)$, with the boundary Dirac term understood by Sobolev duality.

The next result shows that the preceding identity remains valid for solutions of the shifted MIT problem on the complete manifold $M$. No additional geometric or topological assumption on the ends is needed.
\begin{proposition}
	Assume that $\mathcal E\geq0$. For every $\varepsilon>0$ and
	$\varphi\in H^{1/2}(\Sigma,\mathbf S^+)$, the solution
	$\Psi_\varepsilon$ of \eqref{eq:MIT-problem} belongs to
	$H^1(M,\SM)$ and satisfies
	\begin{align}\label{eq:global-Reilly}
		\mathcal B_\Sigma(\Psi_\varepsilon)
		&=\frac12\int_M\norm{\nabla\Psi_\varepsilon}^2
		+\frac12\int_M
		\norm{\nabla\Psi_\varepsilon
			+X^\flat\otimes\Psi_\varepsilon}^2
		+\frac14\int_M\mathcal E\norm{\Psi_\varepsilon}^2
		-\varepsilon^2\int_M\norm{\Psi_\varepsilon}^2.
	\end{align}
	The same conclusion holds for the analogous $D+i\varepsilon$ problem with prescribed $P^-$ trace.
\end{proposition}

\begin{proof}
Completeness of $M$ and compactness of $\Sigma$ provide smooth cutoffs $\chi_R$ with compact support, $0\leq\chi_R\leq1$, equal to one near $\Sigma$, such that $\chi_R\to1$ pointwise and $\Norm{\dd\chi_R}_{L^\infty(M)}\to0$. Since $\Psi_\varepsilon\in H_D^1(M,\SM)$, $\chi_R\Psi_\varepsilon$ belongs to $H^1(M,\SM)$ and has compact
support. Hence \eqref{eq:compact-Reilly}, which extends by density to compactly supported $H^1$ spinors, applies to $\chi_R\Psi_\varepsilon$. Denoting by $\mathcal Q_R$ the sum of the three nonnegative bulk terms in \eqref{eq:compact-Reilly}, we obtain, since $\chi_R=1$ near $\Sigma$,
\[
\mathcal Q_R
=
\mathcal B_\Sigma(\Psi_\varepsilon)
+\Norm{D(\chi_R\Psi_\varepsilon)}_{L^2(M)}^2.
\]
Since $\Psi_\varepsilon\in L^2(M)$, the cutoff properties and the equation $(D-i\varepsilon)\Psi_\varepsilon=0$ imply
\[
D(\chi_R\Psi_\varepsilon)
=
i\varepsilon\chi_R\Psi_\varepsilon
+c(\dd\chi_R)\Psi_\varepsilon
\longrightarrow
i\varepsilon\Psi_\varepsilon
\quad\text{in }L^2(M),
\]
hence 
\[
\mathcal Q_R\to\mathcal B_\Sigma(\Psi_\varepsilon)+\varepsilon^2\Norm{\Psi_\varepsilon}_{L^2(M)}^2.
\]
Fatou's lemma, applied to the nonnegative bulk integrand, now yields
\begin{equation}\label{eq:integrability}
 \nabla\Psi_\varepsilon\in L^2,
 \quad
 \nabla\Psi_\varepsilon+X^\flat\otimes\Psi_\varepsilon\in L^2,
 \quad
 \mathcal E^{1/2}\Psi_\varepsilon\in L^2.
\end{equation}
Together with the already known $L^2$ property of $\Psi_\varepsilon$, the first statement is precisely $\Psi_\varepsilon\in H^1(M,\SM)$.

We may now pass to the limit in the identity itself.  For example,
\[
 \nabla(\chi_R\Psi_\varepsilon)
 =\chi_R\nabla\Psi_\varepsilon
  +\dd\chi_R\otimes\Psi_\varepsilon
 \longrightarrow\nabla\Psi_\varepsilon
 \quad\text{in }L^2,
\]
and the same argument, using the second term in \eqref{eq:integrability}, treats $\nabla(\chi_R\Psi_\varepsilon)+X^\flat\otimes\chi_R\Psi_\varepsilon$. Dominated convergence applies to $\chi_R^2\mathcal E\norm{\Psi_\varepsilon}^2$, while the displayed formula for $D(\chi_R\Psi_\varepsilon)$ treats the Dirac term. This proves \eqref{eq:global-Reilly}.

The proof for the analogous $D+i\varepsilon$ problem with boundary condition $P^-(\Psi_\varepsilon|_\Sigma)=\varphi$, where $\varphi\in H^{1/2}(\Sigma,\mathbf S^-)$, is identical.
\end{proof}

\subsection{Proof of Theorem~\ref{thm:complete-spectral}}

    Set $b_k:=\min_{\Sigma_k}\beta>0$ and fix a boundary component $\Sigma_j$. We shall repeatedly use the following consequence of \eqref{eq:swap} and self-adjointness: every $\psi\in H^{1/2}(\Sigma,\mathbf S)$ satisfies
    \begin{equation}\label{eq:boundary-splitting}
	\operatorname{Re}\int_\Sigma
	\ip{\mathcal D^\Sigma\psi}{\psi}
	=
	2\operatorname{Re}\int_\Sigma
	\ip{\mathcal D^\Sigma P^+\psi}{P^-\psi}.
    \end{equation}

    We first prove that $\mathcal D^{\Sigma_j}$ has trivial kernel. Suppose that $0\neq\Phi\in\ker\mathcal D^{\Sigma_j}$ and write
    $\Phi_\pm=P^\pm\Phi$. By \eqref{eq:swap},
    \[
    \mathcal D^{\Sigma_j}\Phi_\pm=0,
    \]
    and at least one of $\Phi_+$ and $\Phi_-$ is nonzero. Assume first that $\Phi_+\neq0$, extend it by zero to the other boundary components, still denoting the resulting spinor by $\Phi_+$, and solve \eqref{eq:MIT-problem} with boundary datum $\Phi_+$. Identity \eqref{eq:balance} gives
    \[
    \varepsilon^2
    \Norm{\Psi_\varepsilon}_{L^2(M)}^2
    \leq
    \frac{\varepsilon}{2}
    \Norm{\Phi_+}_{L^2(\Sigma_j)}^2.
    \]
    Since $\mathcal D^\Sigma\Phi_+=0$ on $\Sigma$, \eqref{eq:boundary-splitting} shows that the boundary Dirac term vanishes. Denoting by $\mathcal Q_\varepsilon$ the sum of the three nonnegative bulk terms in \eqref{eq:global-Reilly}, we therefore obtain
    \[
    0\leq\mathcal Q_\varepsilon
    =
    \mathcal B_\Sigma(\Psi_\varepsilon)
    +\varepsilon^2\Norm{\Psi_\varepsilon}_{L^2(M)}^2
    \leq
    -\frac{b_j}{2}\Norm{\Phi_+}_{L^2(\Sigma_j)}^2
    +\varepsilon^2\Norm{\Psi_\varepsilon}_{L^2(M)}^2
    \leq
    -\frac{b_j-\varepsilon}{2}
    \Norm{\Phi_+}_{L^2(\Sigma_j)}^2,
    \]
    which is impossible for $0<\varepsilon<b_j$. If instead $\Phi_-\neq0$, the same argument applies to the $D+i\varepsilon$ problem, using the analogous identity~\eqref{eq:balance} and \eqref{eq:boundary-splitting} with $P^+$ and $P^-$ interchanged. Thus $\ker\mathcal D^{\Sigma_j}=0$, which is equivalent to $\ker D^{\Sigma_j}=0$ under the standard identification recalled above.
	
    We now prove the eigenvalue estimate. Since $\mathcal D^{\Sigma_j}$ has trivial kernel, the symmetry of its spectrum and the identification recalled above show that its smallest positive eigenvalue is $\lambda(\Sigma_j)$. Let $\Phi$ satisfy
    \[  
    \mathcal D^{\Sigma_j}\Phi=\lambda\Phi,
    \qquad
    \lambda=\lambda(\Sigma_j)>0,
    \]
    and write $\Phi_\pm=P^\pm\Phi$. By \eqref{eq:swap},
    \[
    \mathcal D^{\Sigma_j}\Phi_\pm=\lambda\Phi_\mp,
    \]
    and self-adjointness yields
    \[
    \Norm{\Phi_+}_{L^2(\Sigma_j)}^2
    =
    \Norm{\Phi_-}_{L^2(\Sigma_j)}^2
    =:a>0.
    \]

    Extend $\Phi_+$ by zero to the other boundary components, still denoting the resulting spinor by $\Phi_+$, and let $\Psi_\varepsilon$ solve \eqref{eq:MIT-problem} with this boundary datum. By \eqref{eq:boundary-splitting},
    \begin{align*}
	\operatorname{Re}\int_\Sigma
	\ip{\mathcal D^\Sigma\Psi_\varepsilon}{\Psi_\varepsilon}
	\leq
	\lambda\left(
	a+\Norm{P^-\Psi_\varepsilon}_{L^2(\Sigma_j)}^2
	\right).
    \end{align*}
    Combining the preceding boundary estimate with \eqref{eq:balance} and \eqref{eq:global-Reilly}, and using $\beta\geq b_k$ on each $\Sigma_k$, we obtain
    \[
    0\leq\mathcal Q_\varepsilon
    \leq
    \left(\lambda-\frac{b_j}{2}\right)
    \left(
    a+\Norm{P^-\Psi_\varepsilon}_{L^2(\Sigma_j)}^2
    \right)
    +\frac{\varepsilon}{2}a,
    \]
    where the nonpositive contributions of the other boundary components have been discarded. Since
    \[
    a+\Norm{P^-\Psi_\varepsilon}_{L^2(\Sigma_j)}^2\geq a>0,
    \]
    the preceding inequality implies
    \[
    \lambda-\frac{b_j}{2}
    \geq
    -\frac{\varepsilon a}
    {2\left(
    	a+\Norm{P^-\Psi_\varepsilon}_{L^2(\Sigma_j)}^2
    	\right)}
    \geq-\frac{\varepsilon}{2}.
    \]
    Letting $\varepsilon$ tend to zero proves the desired estimate.
	
	Suppose now that equality holds for $\Sigma_j$, so that $\lambda=b_j/2$. Retaining all the nonpositive boundary contributions in the preceding computation gives
	\begin{align}\label{eq:small-energy}
		0
		\leq
		\mathcal Q_\varepsilon
		+\frac12\int_{\Sigma_j}
		(\beta-b_j)\norm{\Psi_\varepsilon}^2
		+\frac12\sum_{k\neq j}\int_{\Sigma_k}
		\beta\norm{\Psi_\varepsilon}^2
		\leq
		\varepsilon^2
		\Norm{\Psi_\varepsilon}_{L^2(M)}^2
		\leq\frac{\varepsilon }{2}a.
	\end{align}
	In particular,
	\[
	\Norm{\nabla\Psi_\varepsilon}_{L^2(M)}^2
	\leq\varepsilon a,
	\qquad
	\Norm{\Psi_\varepsilon}_{L^2(\Sigma_j)}^2
	\leq2a.
	\]
	For every connected compact subdomain $M_0$ containing $\Sigma$, the standard trace--Poincar\'e estimate
	\[
	\Norm{\psi}_{L^2(M_0)}
	\leq
	C_{M_0}\left(
	\Norm{\nabla\psi}_{L^2(M_0)}
	+\Norm{\psi}_{L^2(\Sigma_j)}
	\right)
	\]
	therefore gives locally uniform $H^1$ bounds. Rellich compactness, the compactness of the trace, and a diagonal argument provide a sequence $\varepsilon_\ell$ which tends to zero and a spinor $\Psi_0\in H^1_{\mathrm{loc}}(M,\SM)$ such that
	\[
	\Psi_{\varepsilon_\ell}\rightharpoonup\Psi_0
	\quad\text{in }H^1_{\mathrm{loc}}(M),
	\qquad
	\Psi_{\varepsilon_\ell}\longrightarrow\Psi_0
	\quad\text{in }L^2_{\mathrm{loc}}(M)
	\text{ and }L^2(\Sigma).
	\]
	It follows from \eqref{eq:small-energy} that
	\[
	\nabla\Psi_0=0,
	\qquad
	X^\flat\otimes\Psi_0=0,
	\qquad
	\mathcal E\norm{\Psi_0}^2=0.
	\]
	Moreover,
	\[
	P^+\Psi_0|_{\Sigma_j}=\Phi_+,
	\]
	so $\Psi_0$ is nonzero. Since $M$ is connected, its norm is therefore a positive constant. Consequently, $X=0$, and the standard curvature
	identity for a parallel spinor gives $\Ric_h=0$.
	
	Passing to the limit in the boundary terms of \eqref{eq:small-energy} yields
	\[
	\int_{\Sigma_j}(\beta-b_j)\norm{\Psi_0}^2=0,
	\qquad
	\int_{\Sigma_k}\beta\norm{\Psi_0}^2=0
	\quad\text{for every }k\neq j.
	\]
	Since $\norm{\Psi_0}$ is a positive constant, the first equality gives $\beta=b_j$ on $\Sigma_j$, whereas the second one is incompatible with $\beta\geq b_k>0$ on any other boundary component. Thus $\Sigma$ is connected and, since $X=0$,
	\[
	\beta=H=b_j
	\]
	is constant on $\Sigma$.
	
	Finally, since $(M,h)$ is connected and complete with compact boundary, $\Ric_h=0$, and $H=b_j>0$ on $\Sigma$, Kasue's inradius estimate \cite[Theorem~A]{Kasue} implies that $M$ is compact; hence equality cannot occur in the noncompact case.
	
\subsection{Proof of Theorem~\ref{thm:complete-total}}
    
    Fix a boundary component $\Sigma_j$. Using the intrinsic identification of the adapted boundary spinor bundle recalled above, choose, as in \cite[Proposition~2]{BaerTotal}, an integer $r\geq1$ and spinors
    \[
    \varphi_1,\ldots,\varphi_r
    \in C^\infty(\Sigma_j,\mathbf S|_{\Sigma_j}).
    \]
    Since $P^\pm$ are orthogonal projections onto half-rank subbundles, the family may be normalized so that
    \begin{equation}\label{eq:Baer-family}
    	\sum_{\alpha=1}^r\norm{P^\pm\varphi_\alpha}^2=1.
    \end{equation}
    Fixing such a family, set
    \[
    	C(\Sigma_j)
    	:=
    	4\sum_{\alpha=1}^r
    	\Norm{\mathcal D^{\Sigma_j}\varphi_\alpha}_{L^2(\Sigma_j)}
    	\Norm{\varphi_\alpha}_{L^2(\Sigma_j)}.
    \]
	
	For each $\alpha$, extend $P^+\varphi_\alpha$ by zero to the other boundary components and let $\Psi_{\varepsilon,\alpha}$ solve
	\[
	(D-i\varepsilon)\Psi_{\varepsilon,\alpha}=0,
	\qquad
	P^+\Psi_{\varepsilon,\alpha}=P^+\varphi_\alpha.
	\]
	Set
	\[
	I_{\varepsilon,\alpha}
	:=
	\operatorname{Re}\int_\Sigma
	\ip{\mathcal D^\Sigma\Psi_{\varepsilon,\alpha}}
	{\Psi_{\varepsilon,\alpha}}.
	\]
	Since the bulk terms in \eqref{eq:global-Reilly} are nonnegative, $\beta\geq0$, and
	\[
	\sum_{\alpha=1}^r
	\norm{P^+\Psi_{\varepsilon,\alpha}}^2
	=
	\sum_{\alpha=1}^r
	\norm{P^+\varphi_\alpha}^2
	=1
	\quad\text{on }\Sigma_j,
	\]
	we obtain
	\begin{equation}\label{eq:total-Reilly}
		\frac12\int_{\Sigma_j}\beta
		\leq
		\sum_{\alpha=1}^r I_{\varepsilon,\alpha}
		+\varepsilon^2\sum_{\alpha=1}^r
		\Norm{\Psi_{\varepsilon,\alpha}}_{L^2(M)}^2.
	\end{equation}
	
	On every component different from $\Sigma_j$, the contribution to $I_{\varepsilon,\alpha}$ vanishes because the prescribed $P^+$ trace is zero. On $\Sigma_j$, \eqref{eq:boundary-splitting} and \eqref{eq:balance} give
	\begin{align*}
		I_{\varepsilon,\alpha}
		 =
		2\operatorname{Re}\int_{\Sigma_j}
		\ip{\mathcal D^{\Sigma_j}P^+\varphi_\alpha}
		{P^-\Psi_{\varepsilon,\alpha}}                                      
		\leq
		2\Norm{\mathcal D^{\Sigma_j}\varphi_\alpha}_{L^2(\Sigma_j)}
		\Norm{\varphi_\alpha}_{L^2(\Sigma_j)},
	\end{align*}
	while
	\[
	\varepsilon^2
	\Norm{\Psi_{\varepsilon,\alpha}}_{L^2(M)}^2
	\leq
	\frac{\varepsilon}{2}
	\Norm{P^+\varphi_\alpha}_{L^2(\Sigma_j)}^2.
	\]
	Substituting these estimates into \eqref{eq:total-Reilly} and using \eqref{eq:Baer-family}, we find
	\[
	\frac12\int_{\Sigma_j}\beta
	\leq
	2\sum_{\alpha=1}^r
	\Norm{\mathcal D^{\Sigma_j}\varphi_\alpha}_{L^2(\Sigma_j)}
	\Norm{\varphi_\alpha}_{L^2(\Sigma_j)}
	+\frac{\varepsilon}{2}\vol(\Sigma_j).
	\]
	Letting $\varepsilon$ tend to zero gives
	\[
	\int_{\Sigma_j}\beta\leq C(\Sigma_j).
	\]
	By construction, $C(\Sigma_j)$ depends only on the intrinsic Riemannian spin geometry of $\Sigma_j$.


\section{Applications to initial data sets}\label{sec:jang}


Our argument follows the Jang-equation reduction introduced by Schoen and Yau in \cite{SchoenYau}.


\subsection{The complete Jang graph}


Jang's equation for $u$ is
\begin{equation}\label{eq:Jang}
	\left(
	g^{ij}-\frac{u^iu^j}{1+\norm{Du}_g^2}
	\right)
	\left(
	\frac{\nabla^2_{ij}u}{\sqrt{1+\norm{Du}_g^2}}-K_{ij}
	\right)
	=0.
\end{equation}
We use the following compact-boundary version of the capillarity-limit construction. In dimension three, it follows from \cite[Theorems~3.1--3.2]{AEM}; for $4\leq n\leq7$, the same conclusion is obtained by combining the two-sided boundary barriers described in \cite[Section~6.1]{HKKZ} with the almost-minimizing compactness and regularity argument of \cite[Proposition~7 and Corollary~8]{Eichmair}. 

If $3\leq n\leq7$ and
\[
H>|\tr_{\partial\Omega}K|
\quad\text{on }\partial\Omega,
\]
there exist an open set $\Omega_0\subset\Omega$, containing a collar of $\partial\Omega$, and a smooth solution $u$ of \eqref{eq:Jang} on $\Omega_0$, smooth up to $\partial\Omega$, such that $u=0$ there. Moreover, $\partial\Omega_0\setminus\partial\Omega$ is a finite union of smooth closed apparent horizons, and $u$ tends uniformly to $+\infty$ or $-\infty$ along each of its components. Every connected component of the graph meeting $\partial\Omega$ is complete for
\[
\widehat g=g+\dd u^2.
\]

Put $v=\sqrt{1+\norm{Du}_g^2}$ and orient the graph in $(\Omega\times\R,g+\dd t^2)$ by the downward unit normal
\[
\widehat\nu=\frac{Du-\partial_t}{v}.
\]
Its second fundamental form is $\widehat h_{ij}=\nabla^2_{ij}u/v$. Define
\[
	w_i=\frac{u_i}{v},
	\qquad
	q_i=\frac{u^j}{v}
	\bigl(\widehat h_{ij}-K_{ij}\bigr).
\]
The Schoen--Yau identity \cite[equation~(10)]{Eichmair} reads
\begin{equation}\label{eq:SY}
	\widehat R-2\norm{q}_{\widehat g}^2
	+2\widehat\diver q
	=
	2\bigl(\mu-J(w)\bigr)
	+\norm{\widehat h-K}_{\widehat g}^2
	\geq0,
\end{equation}
where the last inequality follows from the dominant energy condition and $\norm{w}_g<1$.

Since $u=0$ on $\partial\Omega$, its tangential differential vanishes there and $\widehat g|_{T\Sigma}=g|_{T\Sigma}$. The boundary computation of Yau \cite[Section~5, equations~(5.6)--(5.10)]{Yau}, see also \cite[equation~(14)]{LiuYau}, gives
\begin{equation}\label{eq:Jang-boundary}
	\widehat H-q(\widehat N)
	=
	\sqrt{1+\norm{Du}_g^2}\,H
	-N(u)P
	\geq
	\sqrt{H^2-P^2}
	=
	\norm{\mathcal H},
\end{equation}
where $N$ and $\widehat N$ are the outward unit normals to $\partial\Omega$ in $(\Omega,g)$ and to the boundary of the graph, respectively, and $P=\tr_\Sigma K$. Since $u=0$ on $\partial\Omega$, we have $Du=N(u)N$ there, and the inequality follows from the untrapped condition.


\subsection{Proofs of the main results}


\begin{proof}[Proof of Theorem~\ref{thm:initial-spectral}]
	Fix a component $\Sigma_j$ and let $M_j$ be the connected component of the Jang graph meeting it. Then $(M_j,\widehat g)$ is a complete spin manifold with compact boundary, and its intrinsic boundary geometry agrees with that induced by $g$. With $X=q^\sharp$, \eqref{eq:SY} and \eqref{eq:Jang-boundary} verify the hypotheses of
	Theorem~\ref{thm:complete-spectral} and give
	\[
	\lambda(\Sigma_j)
	\geq
	\frac12\min_{\Sigma_j}
	\bigl(\widehat H-q(\widehat N)\bigr)
	\geq
	\frac12\min_{\Sigma_j}\norm{\mathcal H}.
	\]
	This proves \eqref{eq:initial-spectral}.
	
	Assume now that equality holds. Both inequalities above are then equalities. The rigidity statement in Theorem~\ref{thm:complete-spectral} shows that $M_j$ is compact, its boundary is connected, $q=0$, and $(M_j,\widehat g)$ carries a nonzero parallel spinor; in particular, $\widehat g$ is Ricci-flat. Any interior boundary component of the graphical domain would force the height function $u$ to be unbounded, contradicting the compactness of $M_j$. The graphical domain is therefore all of $\Omega$; hence $\partial M_j=\Sigma$, which is connected by the rigidity statement.
	
	With $q=0$ and $\widehat R=0$, \eqref{eq:SY} implies that its nonnegative right-hand side vanishes, and hence $\widehat h=K$. Identify $M_j$ with $\Omega$ through the graph projection and consider
	\[
	F:\Omega\longrightarrow M_j\times\R,
	\qquad
	F(x)=(x,u(x)),
	\]
	where the product is endowed with the Ricci-flat Lorentzian metric $\widehat g-\dd t^2$. Using $\widehat g=g+\dd u^2$ and $\widehat\nabla^2u=v^{-2}\nabla_g^2u$, we obtain
	\[
	F^*(\widehat g-\dd t^2)=g,
	\qquad
	\operatorname{II}_F
	=
	v\widehat\nabla^2u
	=
	\frac{\nabla_g^2u}{v}
	=
	\widehat h
	=
	K,
	\]
	where the second fundamental form is computed with respect to the unit normal $v(\partial_t+\widehat\nabla u)$. This gives the required spacelike isometric immersion. When $n=3$, the Ricci-flat metric $\widehat g$ is flat, so the ambient product is locally isometric to Minkowski spacetime.
	
	Finally,
	\[
	\widehat H-q(\widehat N)
	=
	2\lambda(\Sigma)
	=
	\min_\Sigma\norm{\mathcal H}
	\]
	is constant. Together with \eqref{eq:Jang-boundary}, this gives
	\[
	\widehat H-q(\widehat N)
	\geq\norm{\mathcal H}
	\geq\min_\Sigma\norm{\mathcal H},
	\]
	so equality holds throughout and $\norm{\mathcal H}$ is constant on $\Sigma$.
\end{proof}

\begin{proof}[Proof of Theorem~\ref{thm:total}]
	Fix $\Sigma_j$, let $M_j$ be the connected component of the Jang graph meeting $\Sigma_j$, and apply Theorem~\ref{thm:complete-total} to $M_j$ with $X=q^\sharp$. Equations \eqref{eq:SY} and \eqref{eq:Jang-boundary} give
	\[
	\int_{\Sigma_j}\norm{\mathcal H}
	\leq
	\int_{\Sigma_j}\bigl(\widehat H-q(\widehat N)\bigr)
	\leq
	C(\Sigma_j),
	\]
	which proves \eqref{eq:total}.
\end{proof}

\begin{remark}
	The boundary data used here retain less information than the full spacetime Bartnik data \cite{Bartnik97}. In the notation used here, these are $(\Sigma,\gamma,H,\alpha,P)$, where
	\[
	\alpha=K(N,\cdot)|_{T\Sigma}.
	\]
	The corresponding notion of a DEC spin fill-in was considered in \cite{RaulotDEC}. Since $\alpha$ is not prescribed here, every DEC spin fill-in in the sense of \cite{RaulotDEC} is also a DEC spin fill-in in the present sense.
	
	Moreover, whenever the generalized mean curvature used in \cite{RaulotDEC},
	\[
	\mathfrak h:
	=
	H-\sqrt{\norm{\alpha}_\gamma^2+P^2},
	\]
	is positive, the data are untrapped and
	\[
	\mathfrak h
	\leq H-|P|
	\leq\sqrt{H^2-P^2}
	=\norm{\mathcal H}.
	\]
	Thus Corollary~\ref{cor:radius} and Theorem~\ref{thm:total} also give, respectively, upper bounds for the minimum and the integral of $\mathfrak h$.
\end{remark}




\begin{thebibliography}{99}

\bibitem{AEM}
L.~Andersson, M.~Eichmair, and J.~Metzger,
\emph{Jang's equation and its applications to marginally trapped
surfaces}, in Complex Analysis and Dynamical Systems IV, Part 2,
Contemp. Math. \textbf{554} (2011), 13--45.

\bibitem{AmmannLockman}
B.~Ammann and S.~Lockman,
\emph{Rigidity for spin fill-ins with scalar curvature bounded from
below}, arXiv:2608.06935 (2026).

\bibitem{BaerRadius}
C.~B\"ar,
\emph{Dirac eigenvalues and the hyperspherical radius},
J. Eur. Math. Soc. (2026), DOI 10.4171/JEMS/1754.

\bibitem{BaerTotal}
C.~B\"ar,
\emph{Upper bound for the total mean curvature of spin fill-ins},
arXiv:2601.06713v3 (2026).

\bibitem{BaerBallmann}
C.~B\"ar and W.~Ballmann,
\emph{Boundary value problems for elliptic differential operators of first order},
Surveys in Differential Geometry \textbf{17} (2012), 1--78.

\bibitem{Bartnik97}
R.~Bartnik,
\emph{Energy in general relativity},
in \emph{Tsing Hua Lectures on Geometry and Analysis
	(Hsinchu, 1990--1991)},
International Press, Cambridge, MA, 1997, pp.~5--27.

\bibitem{BTW}
S.~Brendle, R.~Tsiamis, and Y.~Wang,
\emph{On fill-ins with scalar curvature bounded from below and an
inequality of Hijazi--Montiel--Rold\'an}, arXiv:2510.17780v2 (2025).

\bibitem{CHZ}
S.~Cecchini, S.~Hirsch, and R.~Zeidler,
\emph{Rigidity of spin fill-ins with non-negative scalar curvature}, arXiv:2404.17533 (2024).

\bibitem{Eichmair}
M.~Eichmair,
\emph{The Jang equation reduction of the spacetime positive energy
theorem in dimensions less than eight},
Comm. Math. Phys. \textbf{319} (2013), 575--593.

\bibitem{Gromov}
M.~Gromov,
\emph{Four lectures on scalar curvature}, in Perspectives in Scalar
Curvature, Vol.~1, World Scientific, Hackensack, NJ, 2023, 1--514.

\bibitem{GrosseNakad}
N.~Gro{\ss}e and R.~Nakad,
\emph{Boundary value problems for noncompact boundaries of
	$\Spin^c$ manifolds and spectral estimates},
Proc. Lond. Math. Soc. (3) \textbf{109} (2014), 946--974.

\bibitem{HMZ}
O.~Hijazi, S.~Montiel, and X.~Zhang,
\emph{Dirac operator on embedded hypersurfaces},
Math. Res. Lett. \textbf{8} (2001), 195--208.

\bibitem{HKKZ}
S.~Hirsch, D.~Kazaras, M.~Khuri, and Y.~Zhang,
\emph{Spectral torical band inequalities and generalizations of the
Schoen--Yau black hole existence theorem},
Int. Math. Res. Not. IMRN \textbf{2024} (2024), 3139--3175.

\bibitem{Jauregui}
J.~L.~Jauregui,
\emph{Fill-ins of nonnegative scalar curvature, static metrics, and
	quasi-local mass},
Pacific J. Math. \textbf{261} (2013), 417--444.

\bibitem{Kasue}
A.~Kasue,
\emph{Ricci curvature, geodesics and some geometric properties of
Riemannian manifolds with boundary},
J. Math. Soc. Japan \textbf{35} (1983), 117--131.

\bibitem{LiuYau}
C.-C.~M. Liu and S.-T.~Yau,
\emph{Positivity of quasi-local mass II},
J. Amer. Math. Soc. \textbf{19} (2006), 181--204.

\bibitem{MantoulidisMiao}
C.~Mantoulidis and P.~Miao,
\emph{Total mean curvature, scalar curvature, and a variational analog
	of Brown--York mass},
Comm. Math. Phys. \textbf{352} (2017), 703--718.

\bibitem{Miao}
P.~Miao,
\emph{Nonexistence of NNSC fill-ins with large mean curvature},
Proc. Amer. Math. Soc. \textbf{149} (2021), 2705--2709.

\bibitem{Raulot}
S.~Raulot,
\emph{The Dirac operator on untrapped surfaces},
Comm. Math. Phys. \textbf{318} (2013), 411--427.

\bibitem{RaulotDEC}
S.~Raulot,
\emph{Nonexistence of DEC spin fill-ins},
Comptes Rendus Math. \textbf{360} (2022), 1049--1054.

\bibitem{SchoenYau}
R.~Schoen and S.-T.~Yau,
\emph{Proof of the positive mass theorem. II},
Comm. Math. Phys. \textbf{79} (1981), 231--260.

\bibitem{ShiTam}
Y.~Shi and L.-F.~Tam,
\emph{Positive mass theorem and the boundary behaviors of compact
	manifolds with nonnegative scalar curvature},
J. Differential Geom. \textbf{62} (2002), 79--125.

\bibitem{ShiWangWei}
Y.~Shi, W.~Wang, and G.~Wei,
\emph{Total mean curvature of the boundary and nonnegative scalar
	curvature fill-ins},
J. Reine Angew. Math. \textbf{784} (2022), 215--250.

\bibitem{Yau}
S.-T.~Yau,
\emph{Geometry of three manifolds and existence of black hole due to
boundary effect},
Adv. Theor. Math. Phys. \textbf{5} (2001), 755--767.

\end{thebibliography}
\end{document}